\documentclass[a4paper,11pt,reqno]{amsart}
\usepackage[a4paper,hscale=0.65,vscale=0.65,centering]{geometry}

\newcommand{\erre}{\mathbb{R}}

\newcommand{\E}{\mathbb{E}}
\newcommand{\EE}{\mathcal{E}}
\renewcommand{\L}{\mathcal{L}}

\renewcommand{\P}{\mathbb{P}}
\renewcommand{\r}{\mathbb{R}}

\newcommand{\ip}[2]{\langle #1,#2 \rangle}
\newcommand{\bip}[2]{\left\langle #1,#2 \right\rangle}
\newcommand{\ds}{\displaystyle}

\newcommand{\sgn}{\mathop{\mathrm{sgn}}\nolimits}

\newcommand{\var}{\mathop{\mathrm{Var}}\nolimits}

\newtheorem{prop}{Proposition}
\newtheorem{thm}[prop]{Theorem}
\newtheorem{coroll}[prop]{Corollary}
\newtheorem{lemma}[prop]{Lemma}

\theoremstyle{definition}
\newtheorem{rmk}[prop]{Remark}

\newtheorem*{ex}{Example}
\newtheorem*{exc}{Example (contd.)}

\begin{document}

\title[Feynman-Kac propagators]{$L^p$ estimates for
  Feynman-Kac propagators with time-dependent reference
  measures}

\author{Andreas Eberle}
\address{Institut f\"ur Angewandte Mathematik, Universit\"at Bonn,
  Endenicher Allee 60, D-53115 Bonn, Germany}
\email{eberle@uni-bonn.de}
\urladdr{http://wiener.iam.uni-bonn.de/$\sim$eberle}

\author{Carlo Marinelli}
\address{Institut f\"ur Angewandte
  Mathematik, Universit\"at Bonn, Endenicher Allee 60, D-53115 Bonn, Germany}
\urladdr{http://www.uni-bonn.de/$\sim$cm788}

\date{2 September 2009}

\begin{abstract}
  We introduce a class of time-inhomogeneous transition operators of
  Feynman-Kac type that can be considered as a generalization of
  symmetric Markov semigroups to the case of a time-dependent
  reference measure. Applying weighted Poincar\'{e} and logarithmic
  Sobolev inequalities, we derive $L^p \to L^p$ and $L^p \to L^q$
  estimates for the transition operators. Since the operators are
  not Markovian, the estimates depend crucially on the value of $p$.
  Our studies are motivated by applications to sequential Markov Chain
  Monte Carlo methods.
\end{abstract}



\thanks{We thank the referee for careful reading and very valuable
  suggestions on the first version of this paper. The work for this
  paper was carried out while the second-named author was visiting the
  Department of Statistics of Purdue University supported by a MOIF
  fellowship. Both authors were partially supported by the
  Sonderforschungsbereich 611, Bonn.}

\maketitle

\section{Introduction}
The purpose of this work is to derive $L^p \to L^p$ and $L^p \to L^q$
bounds for a class of non-Markovian time inhomogeneous transition
operators $q_{s,t}$. These Feynman-Kac type transition operators play
a r\^{o}le in the analysis of sequential MCMC methods, see \cite{DMDJ,
  EM08a}. In the time-homogeneous case, the class of operators
considered here is precisely that of transition functions of symmetric
Markov processes.

In general, let
\begin{equation*}
 \mu_t(x)\ =\ \frac 1{Z_t}\,\exp \left(-\mathcal{H}_t(x)\,\right)\;\mu_0 (x)\,
 ,\qquad t\geq 0 ,
\end{equation*}
denote a family of mutually absolutely continuous probability measures
on a finite set $S$. Here $Z_t$ is a normalization constant, and
$(t,x)\mapsto\mathcal{H}_t(x)$ is a given function on $[0,\infty)
\times S$ that is continuously differentiable in the first variable.
For instance, if $\mathcal{H}_t(x)=t\,\mathcal{H}(x)$ for some
function $\mathcal{H}:S\to \mathbb{R}$, then $(\mu_t)_{t\geq 0}$ is
the exponential family corresponding to $\mathcal{H}$ and $\mu_0$.  We
assume that $S$ is finite to keep the presentation as simple and
non-technical as possible, although most results of this paper extend
to continuous state spaces under standard regularity assumptions.

Note that if $\mathcal{H}_t \equiv 0$ for all $t \geq 0$, then
$\mu_t=\mu_0$ for all $t \geq 0$. In this case, the measures are
invariant for a Markov transition semigroup $(p_t)_{t \geq 0}$, i.e.
\[
p^*_{t-s}\mu_s = \mu_t \qquad \forall t \geq s \geq 0,
\]
if, for example, the generator satisfies a detailed balance condition
w.r.t. $\mu_0$. Here $p^*_{t-s}$ stands for the adjoint of the matrix
$p_{t-s}$, i.e.
\[
(p^*_{t-s} \mu_s)(y) := (\mu_s p_{t-s})(y) = 
\sum_{x \in S} \mu_s(x) p_{t-s}(x,y).
\]
It is well known that in this time-homogeneous case, $L^p$ and $L^p
\to L^q$ bounds for the transition operators $p_t$ follow from
Poincar\'e inequalities (i.e. spectral gap estimates) and logarithmic
Sobolev inequalities w.r.t. the measure $\mu_0$, respectively. We
refer to \cite{SC} and references therein for more background and
results on corresponding bounds for time-homogeneous Markov chains
(see also
\cite{AldFill,BobTet06,ChenMuFa-ergod,DSC-logsob,DiaStr91,LePeWi}).
Such bounds are exploited in the mathematical analysis of Markov Chain
Monte Carlo (MCMC) methods for approximating expectation values w.r.t.
the measure $\mu_0$, see e.g.
\cite{Dia-revolution,DSC-metro,MonteTeta} and references therein, as
well as the above references.

We now introduce the class of non-Markovian, time-inhomogeneous
transition operators for which we will prove corresponding $L^p$ and
$L^p \to L^q$ bounds. Let ${\mathcal L}_t$, $t \geq 0$, be generators
($Q$-matrices) of Markov processes on $S$ satisfying the detailed
balance conditions
\begin{equation}
  \label{eq:db}
  \mu_t(x){\mathcal L}_t(x,y) = \mu_t(y){\mathcal L}_t(y,x)
      \quad \forall\ t\geq0, \; x,y\in S.
\end{equation}
In particular, $\mathcal L_t^\ast\mu_t=0$, i.e.,
\begin{equation}     \label{eq:INV1}
\int \L_tf\,d\mu_t = \sum_{x \in S} (\L_tf)(x)\mu_t(x) = 0
\qquad \mbox{for all }f:S\to\r\mbox{ \ and \ }t\geq 0,
\end{equation}
where
\[
(\L_tf)(x) := \sum_{y \in S} \L_t(x,y)f(y).
\]
We assume that $\mathcal{L}_t(x,y)$ depends continuously on $t$, and
we fix a continuous positive function $t \mapsto \lambda_t$. For $0
\leq s \leq t < \infty$, let $q_{s,t}(x,y)$, $x,\,y\in S$, denote
the solutions of the backward equations
\[
- \frac{\partial}{\partial s} q_{s,t}(x,y) = \lambda_s (\L_s q_{s,t})(x,y)
- H_s(x)q_{s,t}(x,y),
  \qquad s \in [0,t],
\]
with terminal condition $q_{t,t}(x,y)=\delta_{x,y}$, where
\[
H_t(x)\ :=\ -\frac{\partial}{\partial t}\log\mu_t(x)
= \frac{\partial}{\partial t} \mathcal{H}_t
- \int \frac{\partial}{\partial t} \mathcal{H}_t\,d\mu_t
\]
denotes the negative logarithmic time derivative of the measures
$\mu_t$.  Since the state space is finite, the solutions are unique.
For $f: S \to \erre$, $q_{s,t}f$ satisfies the backward equation
\begin{equation}
  \label{eq:BWE}
 - \frac{\partial}{\partial s} q_{s,t}f = \lambda_s \L_s q_{s,t}f - H_sq_{s,t}f,
  \qquad s \in [0,t],
\end{equation}
with terminal condition $q_{t,t}f=f$.

As a consequence of the detailed balance condition (\ref{eq:db}) and
the backward equation (\ref{eq:BWE}), it is not difficult to verify
that the invariance property
\begin{equation}
  \label{eq:InvProp}
  q^*_{s,t}\mu_s = \mu_t
\end{equation}
holds for all $t \geq s \geq 0$, see Proposition \ref{prop:prelim}
below.

Moreover, it can be shown that $q_{s,t}f$ is also the unique
solution of the corresponding forward equation
\begin{equation}     \label{eq:FWE}
  \frac{\partial}{\partial t} q_{s,t}f = q_{s,t}(\lambda_t\L_t f - H_t f),
  \qquad t\in [s, \infty ),
\end{equation}
with initial condition $q_{s,s}f=f$. As a consequence, a
probabilistic representation of $q_{s,t}$ is given by the
Feynman-Kac formula
\begin{equation}       \label{eq:FK}
  (q_{s,t}f)(x) = \E_{s,x} \big[ e^{-\int_s^t H_r(X_r)\,dr} f(X_t)
  \big]\qquad\mbox{for all }x\in S,
\end{equation}
where $(X_t)_{t \geq s}$ is a time-inhomogeneous Markov process
w.r.t. $\P_{s,x}$ with generators $\lambda_t \L_t$ and initial condition
$X_s=x$ $\P_{s,x}$-a.s., see e.g. \cite{GS-II}, \cite{Guli}. Let
\begin{equation}       \label{eq:TSG}
  (p_{s,t}f)(x) = \E_{s,x} \big[  f(X_t)
  \big]
\end{equation}
denote the transition operators of this process.

In Theorems \ref{thm:mainp}, \ref{thm:main2} and \ref{thm:lsi}, and
Corollary \ref{cor:loc} below, we derive $L^p$ and $L^p \to L^q$
bounds for the non-Markovian operators $q_{s,t}$. This is partially
similar to the case of time-homogeneous Markov semigroups, but some
important differences occur. In particular, since the operators
$q_{s,t}$ in general are not contractions on $L^\infty$, the resulting
$L^p$ bounds depend crucially on the value of $p$.

\begin{rmk}
  (i) The non-Markovian transition operators $q_{s,t}$ arise naturally
  in the analysis of sequential Markov Chain Monte Carlo methods. For
  a detailed description of sequential MCMC methods and related
  stochastic processes we refer to \cite{DMDJ,EM08a}.

  (ii) Evolution operators such as $q_{s,t}$, also in continuous time
  and space, have been investigated intensively (see e.g. the
  monograph \cite{Guli} and references therein). However, they are
  usually considered in $L^p$ spaces with respect to a \emph{fixed}
  reference measure. In the applications to sequential MCMC methods we
  are interested in, the time-varying measures $\mu_t$ are given a
  priori, and the analysis on the corresponding $L^p$ spaces is
  crucial. Moreover, a setup with time-varying reference measure is
  more natural in many respects. In particular it provides a
  generalization of the $L^p$ theory of symmetric Markov semigroups.

  (iii) There are several generalizations of symmetric
  time-homogeneous Markov semigroups to the case of time-dependent
  reference measures. One possibility is to consider
  time-inhomogeneous Markov semigroups with infinitesimal generators
  satisfying (\ref{eq:db}). However, these semigroups usually do not
  satisfy the invariance property (\ref{eq:InvProp}). Alternatively,
  there exist time-inhomogeneous Markov processes with transition
  semigroup satisfying (\ref{eq:InvProp}). However, the corresponding
  generators depend on $H_t$ in a non-local way, since increasing the
  mass at one point and decreasing the mass at another point requires
  an additional drift of the process between the points. This
  illustrated in the example below. The third possibility, that we
  consider here, is to replace the Markov semigroup by a Feynman-Kac
  semigroup as defined above. Although these semigroups do not
  correspond to a classical Markov process, they can be approximated
  by a stochastic approach combining Markov Chain Monte Carlo and
  importance sampling concepts, cf.  \cite{DMDJ,EM08a}.
\end{rmk}

\begin{ex} \label{ex:primo} To illustrate our setup and, in
  particular, the last remark, we consider a simple situation where
  the weights of the underlying measure $\mu_t$ vary only at two
  points: suppose that $S=\{0,1,2,\ldots,n\}$ for some $n \in
  \mathbb{N}$, $\mu_0$ is the uniform distribution on $S$, and
  \[
  \mu_t(i) = \mu_0(i) = \frac{1}{n+1}
  \]
  for all $t \geq 0$ and $i=1,2,\ldots,n-1$.  Hence only the weights
  $\mu_t(0)$ and $\mu_t(n)$ are not constant in $t$ and
  $\frac{d}{dt}\mu_t(n)=-\frac{d}{dt}\mu_t(0)$, i.e. mass is
  transferred from $0$ to $n$ and viceversa. We describe three types
  of transition operators satisfying the invariance property
  (\ref{eq:InvProp}) in this situation.

  \noindent (i) Suppose that (in contrast to our setup above) $q_{s,t}$ are the
  transition functions of an ordinary time-inhomogeneous Markov
  process on $S$ with generators $\L_t$ satisfying $\L_t(x,y)=0$
  whenever $|x-y|>1$, i.e. the process only jumps to neighbor sites.
  An elementary computation based on the forward equation shows that
  in this case the invariance property (\ref{eq:InvProp}) holds for
  all $0 \leq s \leq t$ if and only if
  \[
  \L_t(y-1,y) - \L_t(y,y-1) =
  (n+1) \, \frac{d}{dt}\mu_t(0)
  \]
  for all $y \in \{1,2,\ldots,n\}$ and $t \geq 0$. Hence to
  compensate for the change of measure at two points, a global drift
  growing linearly with the distance of the two points is required.
  This is inconvenient for the numerical applications we are
  interested in, see \cite{EM08a}.

  \noindent (ii) A second possibility (consistent with our setup)
  would be to choose
  \[
  q_{s,t}(x,y) = \frac{\mu_t(x)}{\mu_s(x)} \, \delta_{x,y}.
  \]
  This corresponds to the case $\lambda_t=0$ for all $t \geq 0$ in the
  framework introduced above, i.e. the underlying Markov process does
  not move at all. In this case $L^p$ bounds for the operators
  $q_{s,t}$ depend on the $L^\infty$ norm of the relative density
  $\mu_t(x)/\mu_s(x)$, which is also inconvenient for the applications
  we are interested in.

  \noindent (iii) A third possibility (again consistent with our setup
  above) is to choose for $\L_t$ the generator of a Random Walk
  Metropolis Chain with respect to the measure $\mu_t$, i.e.
\begin{equation}     \label{eq:RWM}
\L_t(x,y) =
\begin{cases}
\ds \frac12 \min \Big( \frac{\mu_t(x)}{\mu_s(x)},1 \Big), & \text{if }
|x-y|=1,\\
0, & \text{if }
|x-y|>1,
\end{cases}
\end{equation}
and to define $q_{s,t}$ by (\ref{eq:BWE}). Our first main result,
Theorem \ref{thm:mainp} below, shows that in this case for $p \geq 2$,
the $L^p$ bound
\[
\| q_{s,t}f \|_{L^p(\mu_s)} \leq 2^{1/4} \| f \|_{L^p(\mu_t)}
\]
holds for all $0 \leq s \leq t$ and $f:S \to \erre$ provided
$\lambda_t$ is large enough.
\end{ex}

The remaining content of the paper is organized as follows: in Section
\ref{sec:prima} we collect some properties of the propagators
$q_{s,t}$ that are frequently used in subsequent sections. Sections
\ref{sec:glob} and \ref{sec:improved} deal with $L^p$ bounds under the
assumption that global Poincar\'e inequalities hold. In Section
\ref{sec:loc} we apply the results to derive $L^p$ estimates on a
subset that is invariant w.r.t. the underlying dynamics from
Poincar\'e inequalities on this subset. Finally, in Section
\ref{sec:LS} we prove an $L^p \to L^q$ estimate assuming that a
(time-dependent) logarithmic Sobolev inequality is satisfied.

\section{Preliminaries and notation}     \label{sec:prima}
We shall denote throughout the paper the expectation value of a
function $f:S\to\erre $ with respect to a measure $\nu$ on $S$ by
\[
\ip{f}{\nu} := \int f\,d\nu\ = \sum_{x\in S} f(x)\,\nu (x).
\]
The positive and negative part of a function $f$ are defined,
respectively, by
\[
f^+ = \max(f,0), \qquad f^- = \max(-f,0),
\]
so that $f=f^+-f^-$.

For $t \geq 0$ and $f,\,g:S \to\erre$, the Dirichlet form $\EE_t(f,g)$
corresponding to the self-adjoint operator $\L_t$ on $L^2(S,\mu_t)$ is
given by
\begin{equation}     \label{eq:df}
\EE_t(f,g) = - \int f \L_t g\,d\mu_t = \frac12 \sum_{x,y\in S}
(f(y)-f(x))(g(y)-g(x))\L_t(x,y)\,\mu_t(x).
\end{equation}
We shall also use the shorthand notation $\EE_t(f)=\EE_t(f,f)$.

Note that by the definition of $H_t$ one immediately has
\begin{equation}
 \label{eq:timedep}
 \mu_t(x)\ =\ \exp \left(-\int_0^tH_s(x)\,ds\right)\,\mu_0(x),
\end{equation}
and
 \begin{equation}
 \label{eq:centered}
 \ip{H_t}{\mu_t} \ =\ 0 \qquad \forall t\geq 0.
 \end{equation}
In fact, since $\mu_t(S)=1$ for all $t \geq 0$, one has
\[
 \ip{H_t}{\mu_t} = - \sum_{x \in S} \mu_t(x)
 \frac{\partial}{\partial t} \log \mu_t(x) = - \frac{\partial}{\partial t}
 \sum_{x \in S} \mu_t(x) = 0
\]
As a consequence of (\ref{eq:centered}),
\[
H_t = \frac{\partial}{\partial t}\mathcal{H}_t - \ip{\frac{\partial}{\partial
t}\mathcal{H}_t}{\mu_t}.
\]

In the following proposition we collect some properties of the
operators $q_{s,t}$ which will be used throughout the paper.
\begin{prop}     \label{prop:prelim}
  For all $0\le s \le t\le u$,
   \begin{itemize}
  \item[(i)] $q_{s,t}q_{t,u}=q_{s,u}\ $ (Chapman-Kolmogorov equation).
  \item[(ii)] If $f \geq 0$, then $q_{s,t}f \geq 0$ (Positivity
    preserving property).
  \item[(iii)] If $f \geq 0$, then $q_{s,t}f \leq \exp\left( {\int_s^t
        \max_{x\in S} H_r^-(x)\,dr}\right)\, p_{s,t}f\ $ (Pointwise
    estimate).
  \item[(iv)]  $q^*_{s,t}\mu_s = \mu_t\ $ (Invariance).
  \item[(v)]  $\|q_{s,t}f\|_{L^1(\mu_s)} \leq \|f\|_{L^1(\mu_t)}\ $ ($L^1$ bound).
  \item[(vi)] $\|q_{s,t}f\|_{L^p(\mu_s)} \leq
    \exp\left(\frac{p-1}{p}{\int_s^t \max_{x\in S}
        H_r^-(x)\,dr}\right) \, \|f\|_{L^p(\mu_t)}\ $ (Rough $L^p$
    bound).
  \end{itemize}
\end{prop}
\begin{proof}
  (i) is a consequence of the Markov property of the process
  $(X_t,\P_{s,x})$, and (ii), (iii) are immediate by the
  Feynman-Kac representation (\ref{eq:FK}). Similarly, (iv) is an elementary
  consequence of (\ref{eq:timedep}), (\ref{eq:BWE}), and
  (\ref{eq:INV1}), which imply
  $$\frac{\partial}{\partial s}\langle q_{s,t}f,\mu_s\rangle\ =\
  -\langle H_sq_{s,t}f,\mu_s\rangle -\lambda_s\,\langle
  \mathcal{L}_sq_{s,t}f,\mu_s\rangle\, +\,\langle
  H_sq_{s,t}f,\mu_s\rangle \ =\ 0,$$
  and hence $\langle q_{s,t}f,\mu_s  \rangle
  =\langle q_{t,t}f,\mu_t\rangle =\langle f,\mu_t\rangle$ for all
  $f:S\to\mathbb{R}$ and $s\in [0,t]$.\smallskip\\
    In order to prove (v), take $f
  \geq 0$. Then, by (ii) and (iv),
  \[
  \|q_{s,t}f\|_{L^1(\mu_s)} = \ip{q_{s,t}f}{\mu_s} =
  \ip{f}{\mu_sq_{s,t}} = \ip{f}{\mu_t} = \|f\|_{L^1(\mu_t)}.
  \]
  The general case follows by the decomposition $f=f^+-f^-$ with
  $f^+$, $f^- \geq 0$, and using the linearity of
  $q_{s,t}$.\smallskip\\
  Finally, (vi) is a consequence of (iii) when $p=\infty $, and of (v)
  when $p=1$. The assertion for general $p\in [1,\infty ]$ then
  follows by the Riesz-Thorin interpolation theorem, see
  e.g.~\cite[{\S}1.1.5]{Dav-HK}.
\end{proof}

\section{Global $L^p$ estimates}     \label{sec:glob}
Setting
\[
\mathcal{K}_t := \big\{ f:S \to \erre \; : \; \int f\,d\mu_t =0, \;
f\not\equiv 0 \big\},
\]
let
\[
C_t := \sup_{f \in \mathcal{K}_t}\, \frac1{\EE_t(f)} \int f^2\,d\mu_t
\]
denote the (possibly infinite) inverse spectral gap of
$\mathcal{L}_t$, and define
\[
A_t := \sup_{f \in \mathcal{K}_t}\, \frac1{\EE_t(f)} \int (-H_t)f^2\,d\mu_t,
\qquad
B_t := \sup_{f \in \mathcal{K}_t}\, \frac1{\EE_t(f)}
       \Big| \int_S H_tf\,d\mu_t \Big|^2.
\]
Thus $C_t$, $A_t$ and $B_t$ are the optimal constants in the global
Poincar\'{e} inequalities
\begin{align}
\label{eq:P3} {\rm Var}_{\mu_t}(f) &\leq C_t\cdot
\mathcal{E}_t(f) \qquad \forall f:S \to \erre,\\
  \label{eq:P1}
-\int H_t\, \left( f-\int f\, d\mu_t\right)^2\, d\mu_t &\leq
A_t\cdot
\mathcal{E}_t(f)\qquad \forall f:S\to \erre,\\
 \label{eq:P2}
\left|\int H_t\,  f\, d\mu_t\right|^2 &\leq B_t\cdot
\mathcal{E}_t(f)\qquad \forall f:S \to \erre,
\end{align}
where ${\rm Var}_{\mu_t}$ denotes the variance w.r.t. $\mu_t$.

Our aim in this section is to bound the $L^p \to L^p$ norms of the
operators $q_{s,t}$ in terms of the constants $A_t$, $B_t$ and $C_t$.

\begin{rmk}
  (i) There exist efficient techniques to obtain upper bounds for
  $C_t$, for example the method of canonical paths, comparison methods
  (see e.g. \cite{SC}), as well as decomposition methods (see e.g.
  \cite{JSTV}). Variants of these techniques can be applied to
  estimate $A_t$ and $B_t$ as well. 

  (ii) Clearly, one has
  \begin{align}
    \label{eq:atct} A_t &\leq C_t \cdot\max_{x\in
      S}{H}_t^-(x),\\
    \label{eq:btct} B_t &\leq  C_t \cdot {\rm Var}_{\mu_t}({H}_t)\, ,
  \end{align}
  so an upper bound on $C_t$ yields upper bounds on $A_t$ and $B_t$.
\end{rmk}

\begin{exc}
  In the situation of Example (iii) above, suppose that
  \[
  |H_t(x)| = \frac{\big|\frac{d}{dt}\mu_t(x)\big|}{\mu_t(x)} \leq 1
  \]
  for all $x \in S$. Then one can prove the upper bounds
  \[
  A_t \leq 4(n+1), \qquad B_t \leq 8(n+1)
  \]
  for all $t \geq 0$ (see the Appendix). On the other hand, in this
  case the inverse spectral gap $C_t$ is of order $n^2$.
\end{exc}

Let us start with a basic estimate.
\begin{lemma}
For all $s\ge 0$ and $f:S\to\r $, one has
\begin{equation}     \label{eq:A}
  - \int H_s f^2\,d\mu_s \ \leq\ A_s\, \EE_s(f)
       + 2B_s^{1/2} |\ip{f}{\mu_s}| \EE_s(f)^{1/2}.
\end{equation}
\end{lemma}
\begin{proof}
  Set $\bar{f}_s = f - \ip{f}{\mu_s}$. Then, observing that
  $\ip{H_s}{\mu_s}=0$ and $\EE_s(\bar{f}_s)=\EE_s(f)$, (\ref{eq:P1})
  and (\ref{eq:P2}) imply
  \begin{align*}
    - \int H_s f^2\,d\mu_s &\ =\ - \int H_s (\bar{f}_s^2 + \ip{f}{\mu_s}^2
            + 2\bar{f}_s \ip{f}{\mu_s})\,d\mu_s\\
      &\ \leq\ A_s\, \EE_s(f) + 2 |\ip{f}{\mu_s}|\, B_s^{1/2}\, \EE_s(f)^{1/2},
  \end{align*}
  which proves the claim.
\end{proof}
In the following proposition we establish an integral inequality
for the $L^p(\mu_s)$ norm of $q_{s,t}f$.
\begin{prop}     \label{prop:B}
  Let $p \geq 2$ and assume that
  \[
  \lambda_s > pA_s/4 \qquad \forall s \in [0,t].
  \]
  Then for all $s\in[0,t]$ and for all $f:S\to\erre$,
  \begin{equation}     \label{eq:recurs}
    \bip{|q_{s,t}f|^p}{\mu_s} \ \leq \ \ip{|f|^p}{\mu_t}
        + p(p-1) \int_s^t \frac{B_r}{4\lambda_r-pA_r}
                     \big\langle (q_{r,t}|f|)^{p/2},\mu_r\big\rangle^2\,dr
  \end{equation}
\end{prop}
\begin{proof}
Recalling (\ref{eq:timedep}), the backward equation (\ref{eq:BWE})
allows us to write, for $f:S\to \erre_+$ and $r\in [0,t]$,
\begin{align*}
-\frac{\partial}{\partial r} \int (q_{r,t}f)^p\,d\mu_r &\ =\ p \int
(q_{r,t}f)^{p-1} (\lambda_r\L_rq_{r,t}f-H_rq_{r,t}f)\,d\mu_r
+ \int H_r (q_{r,t}f)^p\,d\mu_r\\
&\ =\ -p \lambda_r \EE_r\big(q_{r,t}f,(q_{r,t}f)^{p-1}\big)
   - (p-1) \int H_r (q_{r,t}f)^p\,d\mu_r,
\end{align*}
where we have used the definition of the Dirichlet form $\EE_r$ in
the second step. Applying the inequality
\begin{equation}     \label{eq:str}
 \EE_r(\phi,\phi^{ p-1 })\ \geq \frac{4(p-1)}{p^{2}}\,
      \EE_r(\phi^{ p/2 })
 \qquad \forall \phi:S \to \erre^+
 \end{equation}
(see e.g. \cite[p.~242]{DeuStr}), we obtain
\[
-\frac{\partial}{\partial r} \int (q_{r,t}f)^p\,d\mu_r \ \leq\  -
\frac{4(p-1)}{p} \lambda_r
\EE_r\big((q_{r,t}f)^{p/2},(q_{r,t}f)^{p/2}\big) - (p-1) \int H_r
(q_{r,t}f)^p\,d\mu_r.
\]
Estimate (\ref{eq:A}) combined with the previous inequality yields
\begin{align*}
- \frac{\partial}{\partial r} \int (q_{r,t}f)^p\,d\mu_r \ \leq&\
- \frac{(p-1)}{p}(4\lambda_r - pA_r) \EE_r\big((q_{r,t}f)^{p/2}\big)\\
&\ + 2(p-1)B_r^{1/2} \big|\ip{(q_{r,t}f)^{p/2}}{\mu_r}\big|
    \EE_r\big((q_{r,t}f)^{p/2}\big)^{1/2},
\end{align*}
hence also, using the elementary inequality $-ax + 2bx^{1/2} \leq
b^2/a$, where $a,b,x \geq 0$,
\[
- \frac{\partial}{\partial r} \int (q_{r,t}f)^p\,d\mu_r \ \leq\
\frac{p(p-1)B_r}{4\lambda_r-pA_r} \ip{(q_{r,t}f)^{p/2}}{\mu_r}^2.
\]
Integrating this inequality from $s$ to $t$ with respect to $r$ we
get, recalling that $q_{t,t}f=f$,
\[
  \bip{(q_{s,t}f)^p}{\mu_s} \leq \ip{f^p}{\mu_t}
        + p(p-1) \int_s^t \frac{B_r}{4\lambda_r-pA_r}
                     \big\langle (q_{r,t}f)^{p/2},\mu_r\big\rangle^2\,dr.
\]
Since $|q_{s,t}f| \leq q_{s,t}|f|$, the claim is obtained applying
the above inequality to the positive function $|f|$.
\end{proof}

We can now prove our first main result.
\begin{thm}     \label{thm:mainp}
  Let $t \geq 0$ and $p \geq 2$. Assume that
  \begin{equation}     \label{eq:*}
  \lambda_s \ \geq\  \frac{p}{4}\, A_s\, +\, \frac{p(p+3)}{4}tB_s
  \end{equation}
  for all $s \in [0,t]$.
  Then we have, for all $s \in [0,t]$ and $f:S\to\r $,
  \begin{itemize}
  \item[(i)] $\|q_{s,t}f\|_{L^p(\mu_s)} \ \leq\ 2^{1/4} \|f\|_{L^p(\mu_t)}$;
  \item[(ii)] $\|q_{s,t}f\|_{L^p(\mu_s)}\ \leq\ \|f\|_{L^p(\mu_t)} +
    2^{1/4}\|f\|_{L^{p/2}(\mu_t)}$.
  \end{itemize}
\end{thm}
\begin{proof}
  By Proposition \ref{prop:prelim} (v) we have that (i) always holds
  for $p=1$. Let us now prove that, for $p \geq 2$, (i) holds provided
  (\ref{eq:*}) is satisfied and (i) holds with $p$ replaced by $p/2$.
  In fact, in this case we have
  \[
  \big\langle (q_{r,t}|f|)^{p/2},\mu_r\big\rangle \leq
  2^{p/8} \ip{|f|^{p/4}}{\mu_t}^2 \leq
  2^{p/8} \ip{|f|^{p/2}}{\mu_t} \qquad \forall r \in [0,t].
  \]
  Then (\ref{eq:recurs}) implies
  \begin{equation*}
  \big\langle |q_{s,t}f|^p,\mu_s\big\rangle \leq \ip{|f|^p}{\mu_t}
     + 2^{p/4} p(p-1) \int_s^t \frac{B_r}{4\lambda_r-pA_r}
            \ip{|f|^{p/2}}{\mu_t}^2\,dr.
  \end{equation*}
  A simple calculation shows that, by (\ref{eq:*}),
  \begin{equation}     \label{eq:roco}
  \int_s^t \frac{pB_r}{4\lambda_r-pA_r}\,dr \leq \frac{1}{p+3}.
  \end{equation}
  Therefore, by the elementary inequality
  \[
  2^{-p/4} + \frac{p-1}{p+3} \leq 1 \qquad \forall p \geq 2,
  \]
  we obtain
  \begin{align*}
    2^{-p/4} \big\langle |q_{s,t}f|^p,\mu_s\big\rangle &\leq
       2^{-p/4} \ip{|f|^p}{\mu_t}
          + \frac{p-1}{p+3} \ip{|f|^{p/2}}{\mu_t}^2\\
    &\leq 2^{-p/4} \ip{|f|^p}{\mu_t}
          + \frac{p-1}{p+3} \ip{|f|^p}{\mu_t}\\
    &\leq \ip{|f|^p}{\mu_t},
  \end{align*}
  thus proving our claim.

  \noindent If (\ref{eq:*}) is satisfied for $p=2$, then
  (i) holds with $p=1$, hence with $p=2$. The Riesz-Thorin
  interpolation theorem then allows to conclude that (i) holds for all
  $p \in [1,2]$.

  \noindent If (\ref{eq:*}) is satisfied for some $p \geq 2$,
  then it also holds for all smaller values of $p$, including $p=2$.
  Choosing $n \in \mathbb{N}$ such that $\tilde{p}:=p 2^{-n} \in
  [1,2]$, we conclude that (i) holds for $\tilde{p}$, hence by
  induction it also holds for $p=\tilde{p} 2^n$. The proof of (i) is
  thus complete.

  Let us now prove (ii): by (\ref{eq:recurs}), (\ref{eq:roco}) and (i)
  we have, noting that $(p-1)/(p+3) \leq 1$ for all $p \geq 2$,
  \begin{align*}
  \big\langle |q_{s,t}f|^p,\mu_s\big\rangle &\leq \ip{|f|^p}{\mu_t}
     + p(p-1) \int_s^t \frac{B_r}{4\lambda_r-pA_r}\,dr
       \sup_{r\in[s,t]} \big\langle (q_{r,t}|f|)^{p/2},\mu_r\big\rangle^2\\
  &\leq \ip{|f|^p}{\mu_t} + \frac{p-1}{p+3}
        \sup_{r\in[s,t]} \big\langle (q_{r,t}|f|)^{p/2},\mu_r\big\rangle^2\\
  &\leq \ip{|f|^p}{\mu_t} + 2^{p/4} \ip{|f|^{p/2}}{\mu_t}^2,
  \end{align*}
  which implies (ii), in view of the elementary inequality
  $(a+b)^{1/p} \leq a^{1/p}+b^{1/p}$, with $a$, $b \geq 0$.
\end{proof}

\begin{exc}
  In the situation of Example (iii) above, provided the assumption on
  $H_t$ made above is satisfied, condition (\ref{eq:*}) is fulfilled
  if
  \[
  \lambda_s \geq \big(1 + 2(p+3)t \big)p(n+1).
  \]
\end{exc}

\section{Improved $L^p$ bounds for functions with $\mu_t$
mean zero}     \label{sec:improved}

It is well known (see e.g. \cite{SC}) that for a time-homogeneous
reversible Markov semigroup $p_t$ with stationary distribution $\mu$,
one has
\[
\| p_tf \|_{L^2(\mu)} \leq e^{-t/C} \| f \|_{L^2(\mu)}
\]
for all $f: S\to \erre$ with $\ip{f}{\mu}=0$, where the exponential
decay rate $C$ is the inverse spectral gap. Similar bounds hold for
other $L^p$ norms with $p \in [2,\infty)$. The purpose of this section
is to prove related bounds for $q_{s,t}f$ if $\ip{f}{\mu_t}=0$.

Let us start with a proposition, which can be seen as an analogue of
Proposition \ref{prop:B}, asserting that an exponential decay of order
$\gamma$ of $\ip{|q_{s,t}f|^{p/2}}{\mu_s}$ implies
exponential decay of the same order for $\ip{|q_{s,t}f|^p}{\mu_s}$.

\begin{prop}
  Let $p \geq 2$ and $\gamma \geq 0$, and assume that
  \begin{equation}     \label{eq:mara}
  \lambda_s \geq \frac{p}{4}\,A_s + \kappa \frac{p(p-1)}{4}\,B_s + 
  \frac{\gamma}{2}\, C_s,
  \qquad \forall s \in [0,t]
  \end{equation}
  for some $\kappa>0$. Then we have, for all $s\in[0,t]$ and
  $f:S\to\erre$,
  \begin{equation}    \label{eq:**}
    \ip{|q_{s,t}f|^2}{\mu_s} \leq e^{-\gamma(t-s)}
    \ip{|f|^2}{\mu_t} + \big(1 + \frac1{\kappa\gamma}\big)
    \big(1-e^{-\gamma(t-s)}\big) \ip{f}{\mu_t}^2
  \end{equation}
  and
  \begin{equation}
    \label{eq:***}
    \ip{|q_{s,t}f|^p}{\mu_s} \leq e^{-\gamma(t-s)} \Big(
    \ip{|f|^p}{\mu_t} + \big(\frac1\kappa+\gamma\big) \int_s^t e^{\gamma(t-r)}
    \ip{|q_{r,t}f|^{p/2}}{\mu_r}^2\,dr \Big)
  \end{equation}
\end{prop}
\begin{proof}
  By (\ref{eq:BWE}), (\ref{eq:timedep}), and the definition of
  $\EE_t$, we obtain, similarly as above,
\begin{align}
  - \frac{\partial}{\partial s} e^{\gamma(t-s)} \int |q_{s,t}f|^2\,d\mu_s
  =& - 2e^{\gamma(t-s)} \lambda_s
         \,\EE_s(q_{s,t}f)\nonumber\\
       & - e^{\gamma(t-s)}
          \int H_s\cdot (q_{s,t}f)^2\,d\mu_s\label{eq:star}\\
       & + \gamma e^{\gamma(t-s)} \int (q_{s,t}f)^2\,d\mu_s.\nonumber
\end{align}
By the Poincar\'e inequality (\ref{eq:P3}), we have
\[
\int (q_{s,t}f)^2\,d\mu_s\ \leq\ C_s\cdot
\EE_s(q_{s,t}f) + \ip{q_{s,t}f}{\mu_s}^2.
\]
Moreover, by (\ref{eq:A}),
\[
-\int H_s\, (q_{s,t}f)^2\,d\mu_s \leq A_s
\EE_s(q_{s,t}f) + 2 B_s^{1/2}
|\ip{q_{s,t}f}{\mu_s}|\cdot
\EE_s(q_{s,t}f)^{1/2},
\]
hence, by (\ref{eq:mara}),
\begin{align*}
  -\frac{\partial}{\partial s} e^{\gamma(t-s)} \int |q_{s,t}f|^2\,d\mu_s
  \ \leq\ &- e^{\gamma(t-s)}\big(2\lambda_s - \gamma C_s
        -A_s\big) \EE_s(q_{s,t}f)\\
  \ \ & +2 e^{\gamma(t-s)} B_s^{1/2} |\ip{q_{s,t}f}{\mu_s}|\,
      \EE_s(q_{s,t}f)^{1/2}\\
  \ \ & +\gamma e^{\gamma(t-s)} \ip{q_{s,t}f}{\mu_s}^2\\
  \ \leq\ \frac{B_s}{2\lambda_s - \gamma C_s - A_s}
        &\; e^{\gamma(t-s)}
        \ip{q_{s,t}f}{\mu_s}^2\
   +\ \gamma e^{\gamma(t-s)} \ip{q_{s,t}f}{\mu_s}^2\\
  \ \leq\ \big(\frac1\kappa + \gamma\big) e^{\gamma(t-s)} \ip{f}{\mu_t}^2.
\end{align*}
Here we have used that $\ip{q_{s,t}f}{\mu_s}=\ip{f}{\mu_t}$ as in
Proposition \ref{prop:prelim}. We obtain (\ref{eq:**}) integrating the
previous inequality with respect to $s$.

Let us now prove (\ref{eq:***}): appealing again to (\ref{eq:BWE}) and
(\ref{eq:timedep}), we obtain, in analogy to the
derivation of (\ref{eq:star}),
\begin{align*}
  -\frac{\partial}{\partial s} e^{\gamma(t-s)} \int |q_{s,t}f|^p\,d\mu_s
  &= -p e^{\gamma(t-s)} \lambda_s
     \EE_s\big(|q_{s,t}f|^{p-1}\sgn(q_{s,t}f),q_{s,t}f\big)\\
  &\quad -(p-1) e^{\gamma(t-s)} \int H_s |q_{s,t}f|^p\,d\mu_s\\
  &\quad +\gamma e^{\gamma(t-s)} \int |q_{s,t}f|^p\,d\mu_s.
\end{align*}
Since for all $\phi:S \to \erre$ and $x$, $y \in S$,
\begin{multline*}
\big(\phi(x)-\phi(y)\big) \big( |\phi(x)|^{p-1} \sgn \phi(x)
- |\phi(y)|^{p-1} \sgn \phi(y) \big)\\
\geq \big(|\phi(x)|-|\phi(y)|\big) \big( |\phi(x)|^{p-1}
- |\phi(y)|^{p-1} \big),
\end{multline*}
taking into account that the off-diagonal terms of
$\mathcal{L}_s(x,y)$ are nonnegative, we obtain by
(\ref{eq:df}) that
\[
\EE_s\big(|q_{s,t}f|^{p-1}\sgn(q_{s,t}f),q_{s,t}f\big) \geq
\EE_s\big(|q_{s,t}f|^{p-1},|q_{s,t}f|\big),
\]
hence also, thanks to (\ref{eq:str}),
\[
\EE_s\big(|q_{s,t}f|^{p-1}\sgn(q_{s,t}f),q_{s,t}f\big) \geq
\frac{4(p-1)}{p^2} \EE_s(\big(|q_{s,t}f|^{p/2}\big).
\]
Proceeding now as in the proof of Proposition \ref{prop:B} we get
\begin{align*}
  -\frac{\partial}{\partial s} e^{\gamma(t-s)} \int |q_{s,t}f|^p\,d\mu_s
  &\leq - e^{\gamma(t-s)}\Big(\frac{4(p-1)}{p}\lambda_s - \gamma C_s
        -(p-1)A_s\Big) \EE_s\big(|q_{s,t}f|^{p/2}\big)\\
  &\quad +2(p-1) e^{\gamma(t-s)} B_s^{1/2}\,
                  \ip{|q_{s,t}f|^{p/2}}{\mu_s}\,
      \EE_s\big(|q_{s,t}f|^{p/2}\big)^{1/2}\\
  &\quad +\gamma e^{\gamma(t-s)} \ip{|q_{s,t}f|^{p/2}}{\mu_s}^2\\
  &\leq \frac{(p-1)^2 B_s}{4\lambda_s(p-1)/p - \gamma C_s -(p-1)A_s}
        e^{\gamma(t-s)} \ip{|q_{s,t}f|^{p/2}}{\mu_s}^2\\
  &\qquad + \gamma e^{\gamma(t-s)} \ip{|q_{s,t}f|^{p/2}}{\mu_s}^2\\
  &\leq \big(\frac1\kappa+\gamma\big) e^{\gamma(t-s)}
        \ip{|q_{s,t}f|^{p/2}}{\mu_s}^2.
\end{align*}
The last estimate holds by (\ref{eq:mara}), since $2(p-1)/p \geq 1$.
We obtain (\ref{eq:***}) integrating the previous inequality with
respect to $s$.
\end{proof}

As a consequence we obtain the following result.
\begin{thm}     \label{thm:main2}
  Let $t$, $\alpha$, $\beta \geq 0$. Then for all $f:S\to\erre$ such that
  $\ip{f}{\mu_t}=0$, we have:
  \begin{itemize}
  \item[(i)] If $\lambda_s \geq \frac 12\,A_s + \alpha C_s$ for all
    $s\in[0,t]$, then
  \[
  \|q_{s,t}f\|_{L^2(\mu_s)} \ \leq\ e^{-\alpha(t-s)}\,
  \|f\|_{L^2(\mu_t)}.
  \]
  \item[(ii)] If $p=2^n$ for some $n\in\mathbb{N}$ and
  \[
  \lambda_s\ \geq\ \frac{p}{4}\,A_s + \beta \frac{p-1}{4}\,B_s
  + \alpha \frac{p}{2}\, C_s
  \qquad \forall s \in [0,t],
  \]
  then
  \[
  \|q_{s,t}f\|_{L^p(\mu_s)}\ \leq\ 
  e^{-\alpha(t-s)} \sqrt{2+(\alpha\beta)^{-1}}
  \,\|f\|_{L^p(\mu_t)}.
  \]
  \end{itemize}
\end{thm}
\begin{proof}
  Since $\ip{q_{s,t}f}{\mu_s}=\ip{f}{\mu_t}=0$, assertion (i) follows
  from (\ref{eq:**}) in the limit $\kappa \downarrow 0$.

  \noindent (ii) We shall prove by induction on $n$ that if
  \[
  \lambda_s \geq \frac{p}{4} A_s + \kappa \frac{p(p-1)}{4} B_s
  + \frac{\gamma}{2} C_s \qquad \forall s \in [0,t]
  \]
  for some $\kappa$, $\gamma \geq 0$, then
  \begin{equation}
    \label{eq:XX}
    \|q_{s,t}f\|_{L^p(\mu_s)} \leq e^{-\gamma(t-s)/p}
    \big(2+\frac{1}{\kappa\gamma}\big)^{1/2-1/p} \|f\|_{L^p(\mu_t)}.
  \end{equation}
  This implies (ii) by choosing $\gamma=\alpha p$ and
  $\kappa=\beta/p$.  For $n=1$, i.e. $p=2$, (\ref{eq:XX}) holds by
  (i). Now suppose (\ref{eq:XX}) holds for $n-1$. Then for $r \in [0,t]$,
  \[
  \ip{|q_{r,t}f|^{p/2}}{\mu_r} \leq e^{-\gamma(t-r)}
  \big( 2+\frac{1}{\kappa\gamma} \big)^{p/4-1}
  \ip{|f|^{p/2}}{\mu_t},
  \]
  hence (\ref{eq:***}) yields
  \begin{align*}
  \ip{|q_{s,t}f|^p}{\mu_s} &\leq e^{-\gamma(t-s)} \Big( \ip{|f|^p}{\mu_t}
  + (\kappa^{-1} + \gamma) \int_s^t e^{\gamma(t-r)}
  \ip{|q_{r,t}f|^{p/2}}{\mu_r}^2\,dr \Big)\\
  &\leq e^{-\gamma(t-s)} \ip{|f|^p}{\mu_t} \Big(1 + 
   \frac{\kappa^{-1}+\gamma}{\gamma}(1-e^{-\gamma(t-s)})
   (2+\kappa^{-1}\gamma^{-1})^{p/2-2} \Big)\\
  &\leq e^{-\gamma(t-s)} \big(2+(\kappa\gamma)^{-1}\big)^{p/2-1}
  \ip{|f|^p}{\mu_t},
  \end{align*}
  and thus (\ref{eq:XX}).
\end{proof}
\begin{rmk}
  For general $p \geq 2$, exponential decay of the $L^p(\mu_s)$ norm
  of $q_{s,t}f$ with $\ip{f}{\mu_t}=0$ follows from the above result by
  the Riesz-Thorin interpolation theorem, in analogy to situations
  already encountered before.
\end{rmk}
\begin{exc}
  In the situation of Example (iii) above, provided the assumption on
  $H_t$ made above is satisfied, one has to choose $\lambda_s$ of
  order $n^2$ in order to guarantee exponential decay of
  $\|q_{s,t}f\|_{L^p(\mu_s)}$.
\end{exc}

\section{$L^p$ estimates on invariant subsets}     \label{sec:loc}
The aim of this section is to show that one can still obtain $L^p$
estimates for the transitions operators $q_{s,t}$ on a subset
$\tilde{S} \subseteq S$ that is invariant w.r.t. the underlying
Markovian dynamics, i.e.
\begin{equation}     \label{eq:invtilde}
\L_t(x,y)=0 \qquad \forall (x,y) \in \tilde{S} \times \tilde{S}^c.
\end{equation}
Instead of Poincar\'e inequalities on $S$, we then only have to assume
corresponding inequalities on the subset $\tilde{S}$. The results
stated below are then a consequence of the global bounds derive above,
and they are relevant for the applications studied in \cite{EM08a}.

Let us define, for $t \geq 0$, the conditional measure
\[
\tilde{\mu}_t(x)=\mu_t(x|\tilde{S}) :=
\frac{\mu_t(\{x\}\cap\tilde{S})}{\mu(\tilde{S})},
\]
and set $\tilde{H}_t:=H_t-\ip{H_t}{\tilde{\mu}_t}$. Note that
$\ip{\tilde{H}_t}{\tilde{\mu}_t}=0$, and
\[
\tilde{\mu}_t\ \propto\ \mu_t\ \propto\ \exp\left( -\int_0^t
H_s\,ds\right)\,\mu_0 \ \propto\ \exp\left( {-\int_0^t
\tilde{H}_s\,ds}\right)\,\mu_0 \qquad \text{on }\tilde{S},
\]
where ``$\propto$'' means that the functions agree up to a
multiplicative constant. Thus the conditional measure $\tilde{\mu}_t$
can be represented in the same way as $\mu_t$ with $H_t$ replaced by
$\tilde{H}_t$. Assumption (\ref{eq:invtilde}) implies that $\L_t$ also
satisfies the detailed balance condition with respect to
$\tilde{\mu}_t$:
\begin{equation}     \label{eq:dbtilde}
\tilde{\mu}_t(x)\L_t(x,y) = \tilde{\mu}_t(y)\L_t(y,x)
\qquad \forall t \geq 0, \quad x,\,y \in S.
\end{equation}
Let
\[
\tilde{\EE}_t(f) = - \int f\,\L_tf\,d\tilde{\mu}_t
= \frac12 \sum_{x,y \in S} (f(y)-f(x))^2 \tilde{\mu}_t \L_t(x,y)
\]
denote the corresponding Dirichlet form on $L^2(S,\tilde{\mu}_t)$.
Note that, by (\ref{eq:invtilde}), only the summands for $x$, $y \in
\tilde{S}$ contribute to the sum.

As a consequence of Theorem \ref{thm:mainp}~(i) and Theorem
\ref{thm:main2} we obtain:
\begin{coroll}     \label{cor:loc}
  Assume that (\ref{eq:invtilde}) holds and that $\L_t$ satisfies the
  inequalities
  \begin{align*}
  \operatorname{Var}_{\tilde{\mu}_t}(f) &\leq \tilde{C}_t
  \tilde{\EE}_t(f),\\
  -\int \tilde{H}_t (f-\ip{f}{\tilde{\mu}_t})^2\,d\tilde{\mu}_t
    &\leq \tilde{A}_t \tilde{\EE}_t(f)\\
  \left|\int \tilde{H}_t  f\,d\tilde{\mu}_t\right|^2 &\leq \tilde{B}_t
  \tilde{\EE}_t(f)
  \end{align*}
  for all $f:S \to \erre$. Then the following assertions hold true for
  all $t$, $\alpha$, $\beta \geq 0$:
  \begin{itemize}
  \item[(i)] Let $p \geq 2$. If
  \[
  \lambda_s \ \geq\ \frac{p}{4} \tilde{A}_s + \frac{p(p+3)}{4} t \tilde{B}_s
  \]
  for all $s\in [0,t]$, then
  \[
  \|q_{s,t}f\|_{L^p(\tilde\mu_s)}\ \leq\  2^{1/4}\,
  \frac{\mu_t(\tilde{S})}{\mu_s(\tilde{S})}\, \|f\|_{L^p(\tilde\mu_t)}
  \]
  for all $f:S \to \erre$ and $s\in [0,t]$.
  \item[(ii)] If
  \[
  \lambda_s\ \geq \ \frac12 \tilde{A}_s + \alpha \tilde{C}_s
  \]
  for all $s\in [0,t]$, then
  \[
  \|q_{s,t}f\|_{L^2(\tilde\mu_s)}\ \leq \ e^{-\alpha(t-s)}
  \,\frac{\mu_t(\tilde{S})}{\mu_s(\tilde{S})}\, \|f\|_{L^2(\tilde\mu_t)}
  \]
  for all $f:S \to \erre$ with $\ip{f}{\tilde{\mu}_t}=0$.
  \item[(iii)] If $p=2^n$ and
  \[
  \lambda_s \ \geq\ \frac{p}{4} \tilde{A}_s
    + \beta \frac{p-1}{4} \tilde{B}_s
    + \alpha\frac{p}{2} \tilde{C}_s
  \]
  for all $s\in [0,t]$, then
  \[
  \|q_{s,t}f\|_{L^p(\tilde\mu_s)}\ \leq\
  e^{-\alpha(t-s)}\,
  (2+1/\alpha\beta)^{1/2}\,\frac{\mu_t(\tilde{S})}{\mu_s(\tilde{S})}
  \,\|f\|_{L^p(\tilde\mu_t)}
  \]
  for all $f:S \to \erre$ with $\ip{f}{\tilde{\mu}_t}=0$.
  \end{itemize}
\end{coroll}
\begin{proof}
  Let $\tilde{q}_{s,t}$ denote the transition operators defined via
  the backward equation (\ref{eq:BWE}) with $\tilde{H}_t$ replacing
  $H_t$. By the detailed balance condition (\ref{eq:dbtilde}) and the
  assumptions, the operators $\tilde{q}_{s,t}$ satisfy the $L^p$
  bounds from the previous sections with $\tilde{H}_t$ replacing
  $H_t$, under the conditions on $\lambda_s$, $\tilde{A}_s$,
  $\tilde{B}_s$ and $\tilde{C}_s$ stated above. Now note that by a
  simple calculation based on (\ref{eq:timedep}),
  \[
  H_t(x) = \tilde{H}_t(x) + \ip{H_t}{\tilde\mu_t} = 
  \tilde{H}_t(x) + h_t(\tilde{S}) \qquad \forall x \in \tilde{S},
  \]
  where
  \[
  h_t(\tilde{S}) = -\frac{d}{dt} \log \mu_t(\tilde{S}).
  \]
  Hence for any function $f:S \to \erre$ and for all $x\in\tilde{S}$,
  we have
  \[
  q_{s,t}f(x)\ =\ e^{-\int_s^t h_r(\tilde{S})\,dr} \tilde{q}_{s,t}f(x)
  \ =\ \frac{\mu_t(\tilde{S})}{\mu_s(\tilde{S})}
  \,\tilde{q}_{s,t}f(x).
  \]
  The assertions now follow applying Theorem \ref{thm:mainp}~(i) and
  Theorem \ref{thm:main2} to $\tilde{q}_{s,t}f$.
\end{proof}
In particular, it is worth pointing out that sufficiently strong
mixing properties on the component can make up for an increase of the
weight of the component as long as one is only looking for bounds for
$q_{s,t}f$ on functions $f$ such that $\ip{f}{\tilde{\mu}_t}=0$.

\section{Logarithmic Sobolev inequalities and $L^p \to L^q$
estimates}     \label{sec:LS}
We finally obtain an $L^p \to L^q$ estimate for $q_{s,t}$ from
logarithmic Sobolev inequalities for the Dirichlet forms $\EE_t$, by
an adaptation of the classical argument that a log Sobolev inequality
implies hypercontractivity (see e.g.  \cite{Gross-LNM} or
\cite[{\S}6.1.14]{DeuStr}). This generalizes well-known results for
time-homogeneous Markov chains, for which we refer to e.g.
\cite{DSC-logsob,SC}, to the time-inhomogeneous setting.
\begin{thm}     \label{thm:lsi}
  Suppose that each of the measures $\mu_t$, $t \geq 0$, satisfies a
  logarithmic Sobolev inequality with constant $C^{LS}_t>0$, i.e.
  \begin{equation}     \label{eq:LSI}
  \int f^2 \log \left(\frac{f}{\|f\|_{L^2(\mu_t)}}
  \right)^2\,d\mu_t\
      \leq\ C_t^{LS}\cdot \EE_t(f)
  \end{equation}
  for all $t \geq 0$ and $f:S \to \erre$. Then, for $1 < p \leq q <
  \infty$, one has
  \[
  \|q_{s,t}f\|_{L^q(\mu_s)}\ \leq\ \exp\left({\int_s^t \max
  H_r^-\,dr}\right)\,
  \|f\|_{L^p(\mu_t)}
  \]
  for all $f:S \to \erre$ and $0 \leq s \leq t$ such that
  \[
  \int_s^t \frac{\lambda_r}{C_r^{LS}}\, dr\ \ge\ \frac 14 \log \frac{q-1}{p-1}\  .
  \]
\end{thm}
\begin{proof}
Let us set
\[
\bar{q}_{s,t}f(x) := e^{-\int_s^t \max_{x \in S} H_r^-(x)\,dr} q_{s,t}f(x),
\qquad 0 \leq s \leq t,
\]
which satisfies the backward equation
\[
-\frac{\partial}{\partial s} \bar{q}_{s,t}f = \lambda_s \L_s
\bar{q}_{s,t}f - (H_s + \max H_s^-) \bar{q}_{s,t}f.
\]
Let $p:[0,t] \to ]1,+\infty[$ be a continuously differentiable
function. By computations similar to those carried out in the proof of
Proposition \ref{prop:B}, we obtain, noting that $H_s+\max H_s^- \geq 0$,
\begin{align*}
-{p_s} \|\bar{q}_{s,t}f\|^{p_s-1}_{L^{p_s}(\mu_s)}
   \frac{\partial}{\partial s} \|\bar{q}_{s,t}f\|_{L^{p_s}(\mu_s)}
   \ =&\ -\frac{\partial}{\partial s} \int (\bar{q}_{s,t}f)^{p_s}\,d\mu_s\\
= &\ -p_s\lambda_s\cdot  \EE_s\big((\bar{q}_{s,t}f)^{p_s-1},%
                   \bar{q}_{s,t}f\big)\\
 &\ -(p_s-1) \int (H_s+\max H_s^-) (\bar{q}_{s,t}f)^{p_s}\,d\mu_s\\
 &\ -p_s' \int (\bar{q}_{s,t}f)^{p_s} \log (\bar{q}_{s,t}f)\,d\mu_s\\
\leq\ -4\,&\frac{p_s-1}{p_s}\, \lambda_s\,
\EE_s\big((\bar{q}_{s,t}f)^{p_s/2}\big)
      -\frac{p_s'}{p_s} \int (\bar{q}_{s,t}f)^{p_s} \log\, (\bar{q}_{s,t}f)^{p_s}\,d\mu_s
\end{align*}
for all $f:S\to \erre_+$, where $p'_s:=dp_s/ds$.  Choosing
\[
p_s=1+(p-1)\, \exp \left({4\int_s^t \lambda_r/C_r^{LS}\,dr}\right),
\]
we have $p'_s=-4(p_s-1)\,\lambda_s/C_s^{LS}$, hence
the log Sobolev inequality (\ref{eq:LSI}) implies
\[
-\frac{\partial}{\partial s} \|\bar{q}_{s,t}f\|_{L^{p_s}(\mu_s)}
\leq 0
\]
for all $s \in ]0,t[$. Therefore we can conclude
\[
\|q_{s,t}f\|_{L^{p_s}(\mu_s)} = e^{\int_s^t \max H_r^-\,dr}
   \|\bar{q}_{s,t}f\|_{L^{p_s}(\mu_s)}
\leq e^{\int_s^t \max H_r^-\,dr}
   \|f\|_{L^{p}(\mu_t)}
\]
for all $s \in [0,t]$.
\end{proof}


\section*{Appendix}
In this appendix we prove bounds for the constants $A_t$, $B_t$, $C_t$
in the situation of Example (iii), for a fixed $t \geq
0$. Let us briefly recall the setup: we have $S=\{0,1,\ldots,n\}$,
$\mu_0$ is the uniform distribution on $S$, $\mu_t(i)=\mu_0(i)$ for
all $1 \leq 1 \leq n-1$, and $\L_t$ is defined by (\ref{eq:RWM}).
Denoting the derivative of $\mu_t$ with respect to time by $\mu'_t$,
we have $\mu'_t(i)=0$ for all $1 \leq 1 \leq n-1$ and
$\mu'_t(n)=-\mu'_t(0)$. We are going to assume, without loss of
generality, that $\mu'_t(0) \geq 0$. Then we have
\begin{equation}     \label{eq:hti}
-H_t(i) = \frac{\mu'_t(i)}{\mu_t(i)} =
\begin{cases}
  0, & 1 \leq i \leq n-1,\\[4pt]
\ds  \frac{\mu'_t(0)}{\mu_t(0)} \geq 0, & i=0,\\[12pt]
\ds  \frac{\mu'_t(n)}{\mu_t(n)} \leq 0, & i=n.
\end{cases}
\end{equation}
In this situation we can prove the following estimates for $n \in
\mathbb{N}$.
\begin{lemma}
  Let $A_t$, $B_t$ and $C_t$ be the constants defined in
  (\ref{eq:P3})-(\ref{eq:P2}). Then one has
  \begin{align*}
    A_t &\leq -4 H_t(0) (n+1)\\
    B_t &\leq 4\big( H_t(0)^2 + H_t(n)^2 \big) (n+1)\\
    \frac{(n-4)^4}{48(n+1)^2} \leq C_t &\leq n \, \max
    \left( \frac{n+1}{2},2 \right) \qquad \forall n \geq 4.
  \end{align*}
\end{lemma}
\begin{proof}
  To derive the upper bound for $A_t$, we observe that by
  (\ref{eq:df}) and (\ref{eq:db}) we have
  \begin{equation}    \label{eq:etf}
    \EE_t(f) = \sum_{i=0}^{n-1} (f(i+1)-f(i))^2 a_t(i)
  \end{equation}
  for all $f: S \to \erre$ and $t \geq 0$, where
\[
a_t(i) = \mu_t(i) \L_t(i,i+1) = \frac12 \min \big(
\mu_t(i),\mu_t(i+1) \big),
\]
and, by (\ref{eq:hti}),
\begin{equation}     \label{eq:lhs}
  - \int H_t (f-\ip{f}{\mu_t})^2\,d\mu_t \leq
  -H_t(0) (f(0)-\ip{f}{\mu_t})^2 \mu_t(0).
\end{equation}
Moreover, by (\ref{eq:etf}), we have
\begin{equation}     \label{eq:b1}
  \begin{split}
  (f(0)-\ip{f}{\mu_t})^2 &= \Big( \sum_{k=0}^n (f(k)-f(0)) \mu_t(k) \Big)^2\\
  &= \Big( \sum_{i=0}^{n-1} (f(i+1)-f(i)) \, \sum_{k=i+1}^n \mu_t(k) \Big)^2\\
  &\leq \EE_t(f) \sum_{i=0}^{n-1} \frac{1}{a_t(i)}
   \Big( \sum_{k=i+1}^n \mu_t(k) \Big)^2.
  \end{split}
\end{equation}
Noting that
\begin{equation}     \label{eq:b23}
a_t(i) = 
\begin{cases}
\ds \frac12 \min \Big( \frac{1}{n+1},\mu_t(0) \Big), & i=0,\\[6pt]
\ds \frac{1}{2(n+1)}, & i=1,\ldots,n-2,\\[12pt]
\ds \frac12 \min \Big( \frac{1}{n+1},\mu_t(n) \Big), & i=n-1,\\
\end{cases}
\end{equation}
(\ref{eq:lhs}) and (\ref{eq:b1}) imply
\begin{align*}
  A_t &\leq -H_t(0) \sum_{i=0}^{n-1} \frac{1}{a_t(i)}
       \Big( \sum_{k=i+1}^n \mu_t(k) \Big)^2 \mu_t(0)\\
      &\leq -2H_t(0) \mu_t(0) \Big( \max(n+1,\mu_t(0)^{-1}) + (n-2)(n+1)\\
      &\phantom{\leq -2H_t(0) \mu_t(0) \Big( \quad}
          + \max(n+1,\mu_t(n)^{-1}) \mu_t(n) \Big)\\
      &\leq -2H_t(0) \big( n(n+1)\mu_t(0) + 2 \big)
       \leq -4H_t(0) \, (n+1).
\end{align*}
The upper bound for $B_t$ can be obtained in a similar way: since
$\ip{H_t}{\mu_t}=0$, we have
\begin{align*}
  \Big| \int H_tf\,d\mu_t \Big| &= \Big| \int H_t (f-\ip{f}{\mu_t}) \,d\mu_t
                                   \Big|\\
  &\leq \Big| \int H_t^2 (f-\ip{f}{\mu_t})^2 \,d\mu_t \Big|\\
  &= H_t(0)^2 (f(0)-\ip{f}{\mu_t})^2\,\mu_t(0)
   + H_t(n)^2 (f(n)-\ip{f}{\mu_t})^2\,\mu_t(0)\\
  &\leq 4 \big( H_t(0)^2 + H_t(n)^2 \big) (n+1)
\end{align*}
by an analogous computation as above.

To prove the upper bound for $C_t$ note that, for $f: S \to \erre$ and
$0 \leq k \leq \ell \leq n$, we have
\[
\big( f(\ell) - f(k) \big)^2 = \Big( \sum_{i=k}^{\ell-1} (f(i+1)-f(i)) \Big)^2
\leq (\ell-k) \sum_{i=k}^{\ell-1} \big( f(i+1)-f(i) \big)^2.
\]
Hence, for $t \geq 0$,
\begin{align*}
\var_{\mu_t}(f) &= \frac12 \sum_{k,\ell=0}^n \big( f(\ell)-f(k) \big)^2
                  \mu_t(k)\,\mu_t(\ell)\\
&= \sum_{k < \ell} \big( f(\ell)-f(k) \big)^2 \mu_t(k)\,\mu_t(\ell)\\
&\leq \sum_{i=0}^{n-1} \big( f(i+1)-f(i) \big)^2 \sum_{k=0}^i
      \sum_{\ell=i+1}^n (\ell-k) \mu_t(k)\,\mu_t(\ell)\\
&\leq n \sum_{i=0}^{n-1} \big( f(i+1)-f(i) \big)^2 \,
      \mu_t(\{0,1,\ldots,i\}) \, \mu_t(\{i+1,i+2,\ldots,n\})\\
&\leq n \, \max\big((n+1)/2,2\big) \, \EE_t(f).
\end{align*}
The last estimate holds by (\ref{eq:etf}), (\ref{eq:b23}),
and because
\[
\mu_t(\{0,1,\ldots,i\}) \, \mu_t(\{i+1,i+2,\ldots,n\}) \leq \frac14
\qquad \forall 0 \leq i \leq n.
\]
We have thus proved that $C_t \leq n\,\max\big((n+1)/2,2\big)$.

Conversely, choosing $f(i)=i$ for $1 \leq i \leq n-1$, $f(0)=1$, and
$f(n)=n-1$, we have
\[
\EE_t(f) = \sum_{i=1}^{n-1} a_t(i) = \frac{n-1}{2(n+1)} \leq \frac12
\]
by (\ref{eq:etf}) and (\ref{eq:b23}), and
\begin{align*}
\var_{\mu_t}(f) &\geq \sum_{k=1}^{n-1} \sum_{\ell=k+1}^{n-1}
                     (\ell-k)^2 \mu_t(k) \, \mu_t(\ell)\\
&\geq \frac{1}{8(n+1)^2} \sum_{k=1}^{n-2} \sum_{m=1}^n m^2
      \geq \frac{(n-4)^4}{96(n+1)^2},
\end{align*}
which proves the lower bound for $C_t$.
\end{proof}

\bibliographystyle{amsplain}
\bibliography{mcmc,ref}

\def\cprime{$'$}
\providecommand{\bysame}{\leavevmode\hbox to3em{\hrulefill}\thinspace}
\providecommand{\MR}{\relax\ifhmode\unskip\space\fi MR }
\providecommand{\MRhref}[2]{%
  \href{http://www.ams.org/mathscinet-getitem?mr=#1}{#2}
}
\providecommand{\href}[2]{#2}
\begin{thebibliography}{10}

\bibitem{AldFill}
D.~Aldous and J.~Fill, \emph{Reversible {M}arkov chains and random walks on
  graphs}, Monograph in preparation.

\bibitem{BobTet06}
S.~G. Bobkov and P.~Tetali, \emph{Modified logarithmic {S}obolev inequalities
  in discrete settings}, J. Theoret. Probab. \textbf{19} (2006), no.~2,
  289--336. \MR{MR2283379 (2007m:60215)}

\bibitem{ChenMuFa-ergod}
Mu-Fa Chen, \emph{Eigenvalues, inequalities, and ergodic theory},
  Springer-Verlag, London, 2005. \MR{MR2105651 (2005m:60001)}

\bibitem{Dav-HK}
E.~B. Davies, \emph{Heat kernels and spectral theory}, Cambridge University
  Press, Cambridge, 1990. \MR{MR1103113 (92a:35035)}

\bibitem{DMDJ}
P.~Del~Moral, A.~Doucet, and A.~Jasra, \emph{Sequential {M}onte {C}arlo
  samplers}, J. R. Statist. Soc. B \textbf{68} (2006), no.~3, 411--436.
  \MR{MR1819122 (2002k:60013)}

\bibitem{DeuStr}
J.-D. Deuschel and D.~W. Stroock, \emph{Large deviations}, Academic Press Inc.,
  Boston, MA, 1989. \MR{MR997938 (90h:60026)}

\bibitem{Dia-revolution}
P.~Diaconis, \emph{The {M}arkov chain {M}onte {C}arlo revolution}, Bull. Amer.
  Math. Soc. (N.S.) \textbf{46} (2009), no.~2, 179--205. \MR{MR2476411}

\bibitem{DSC-logsob}
P.~Diaconis and L.~Saloff-Coste, \emph{Logarithmic {S}obolev inequalities for
  finite {M}arkov chains}, Ann. Appl. Probab. \textbf{6} (1996), no.~3,
  695--750. \MR{MR1410112 (97k:60176)}

\bibitem{DSC-metro}
\bysame, \emph{What do we know about the {M}etropolis algorithm?}, J. Comput.
  System Sci. \textbf{57} (1998), no.~1, 20--36, 27th Annual ACM Symposium on
  the Theory of Computing (STOC'95) (Las Vegas, NV). \MR{MR1649805
  (2000b:68094)}

\bibitem{DiaStr91}
P.~Diaconis and D.~Stroock, \emph{Geometric bounds for eigenvalues of {M}arkov
  chains}, Ann. Appl. Probab. \textbf{1} (1991), no.~1, 36--61. \MR{MR1097463
  (92h:60103)}

\bibitem{EM08a}
A.~Eberle and C.~Marinelli, \emph{Quantitative approximations of evolving
  probability measures and sequential {M}arkov {C}hain {M}onte {C}arlo
  methods}, Preprint, 2009.

\bibitem{GS-II}
{\u{I}}.~{\=I}. G{\={\i}}hman and A.~V. Skorohod, \emph{The theory of
  stochastic processes. {II}}, Springer-Verlag, New York, 1975.

\bibitem{Gross-LNM}
L.~Gross, \emph{Logarithmic {S}obolev inequalities and contractivity properties
  of semigroups}, Dirichlet forms (Varenna, 1992), Lecture Notes in Math., vol.
  1563, Springer, Berlin, 1993, pp.~54--88. \MR{MR1292277 (95h:47061)}

\bibitem{Guli}
A.~Gulisashvili and J.~A. van Casteren, \emph{Non-autonomous {K}ato classes and
  {F}eynman-{K}ac propagators}, World Scientific Publishing Co., 2006.
  \MR{MR2253111 (2008b:60161)}

\bibitem{JSTV}
M.~Jerrum, J.-B. Son, P.~Tetali, and E.~Vigoda, \emph{Elementary bounds on
  {P}oincar\'e and log-{S}obolev constants for decomposable {M}arkov chains},
  Ann. Appl. Probab. \textbf{14} (2004), no.~4, 1741--1765. \MR{MR2099650
  (2005i:60139)}

\bibitem{LePeWi}
D.~A. Levin, Y.~Peres, and E.~L. Wilmer, \emph{Markov chains and mixing times},
  American Mathematical Society, Providence, RI, 2009. \MR{MR2466937}

\bibitem{MonteTeta}
R.~Montenegro and P.~Tetali, \emph{Mathematical aspects of mixing times in
  {M}arkov chains}, Found. Trends Theor. Comput. Sci. \textbf{1} (2006), no.~3,
  x+121. \MR{MR2341319 (2009g:68100)}

\bibitem{SC}
L.~Saloff-Coste, \emph{Lectures on finite {M}arkov chains}, Lectures on
  probability theory and statistics (Saint-Flour, 1996), Lecture Notes in
  Math., vol. 1665, Springer, Berlin, 1997, pp.~301--413. \MR{MR1490046
  (99b:60119)}

\end{thebibliography}

\end{document}